\documentclass[12pt,reqno]{amsart}
\usepackage[margin=1in]{geometry}
\usepackage{amsthm, amsmath}
\usepackage{tikz, tikz-cd}
\usepackage{hyperref}
\usepackage{enumerate}
\usetikzlibrary{calc}
\numberwithin{equation}{section}
\newtheorem{theorem}{Theorem}[section]
\newtheorem{lemma}[theorem]{Lemma}
\newtheorem{corollary}[theorem]{Corollary}

\newtheorem{proposition}[theorem]{Proposition}

\theoremstyle{definition}

\newtheorem{definition}[theorem]{Definition}

\newtheorem{remark}[theorem]{Remark}

\newcommand{\R}{\mathbb{R}}

\definecolor{urlcolor}{rgb}{0,.145,.698}
\definecolor{linkcolor}{rgb}{.01,0.01,0.51}
\definecolor{citecolor}{rgb}{.12,.54,.11}
    
\hypersetup{
      breaklinks=true,
      colorlinks=true,
      urlcolor=urlcolor,
      linkcolor=linkcolor,
      citecolor=citecolor,
      }

\usepackage{caption}
\begin{document}

\title{Boundary measurements of positive networks on a cylinder of maximal rank 2 and 3}
\author{David Whiting}\address{Department of Mathematics, Michigan State University, 
East Lansing, MI 48824, USA}

\begin{abstract}
Boundary measurement matrices associated to networks on a plane correspond to certain totally nonnegative Grassmannians as shown previously by A. Postnikov. In this paper, we look to generalize this result by categorizing the boundary measurements associated to networks on a cylinder of maximal rank 2 and 3. In particular, we show that the maximal rank 3 matrices associated to networks on a cylinder are precisely the matrices in which every odd-dimensional minor is nonnegative.
\end{abstract}
\maketitle

\section{Introduction}

Networks on an oriented surface are directed graphs with positive weights on the directed arcs. The sources and sinks of a network will be the vertices on the boundary of the surface. Given a network, we can construct a boundary measurement matrix. See figure $\ref{fig:network_example}$ for an example. Networks on a disk have been previously described by Postnikov in \cite{Postnikov}. More precisely, given a network with $n$ sinks and $m$ sources, it was shown that such a network can be embedded in a disk if and only if every minor of its associated $m \times (m+n)$ boundary measurement matrix is nonnegative. In this paper, we extend this work in the natural way to certain networks in a cylinder, or equivalently, an annulus. 

It is important to note that the definition of a boundary measurement matrix differs in literature. We will define a matrix $M$ using the definition in \cite{GSV-2}, in which we realize $M$ as a $m \times n$ matrix instead of a $m \times (m+n)$ matrix. One can recover the $m \times (m+n)$ boundary measurement matrix $\Tilde{M}$ defined in \cite{Postnikov} from $M$ by adding the columns of an identity matrix of the appropriate size so that the maximal minors of $\Tilde{M}$ are the minors of $M$ (plus the minor corresponding to the embedded identity matrix). The precise definitions we will use are provided in section \ref{sec:prelim}.

There is a very deep connection between these networks, cluster algebras, and integrable systems. For instance, the work by Postnikov in \cite{Postnikov} finds a nice way to describe certain families of cluster coordinates by certain Grassmannian elements. In \cite{GSV-1}, the authors defined a family of Poisson brackets on the boundary measurement matrices of these networks on disks. They generalize this work to an annulus (or cylinder) in \cite{GSV-2}. In particular, they proved that a boundary measurement matrix $M$ satisfies the Sklyanin r-matrix Poisson relations. A non-commutative generalization of this result is provided in \cite{noncommutative}, which is a continuation upon previous results related to discrete completely integrable dynamical systems (see \cite{O1}, \cite{DK}, \cite{DF}, \cite{BR1}) and non-commutative cluster algebras (see \cite{BR2}, \cite{GK}).

In this paper, we consider networks in a cylinder/annulus with an additional restriction; all edges are directed from one boundary circle to the other one. Under these conditions, we describe the set of all boundary measurement matrices of maximal rank 2 and 3.

The rank $2$ case is described using the function $cvar$, which we call the cyclic sign variation function defined in section \ref{sec:prelim}. This function counts the number of sign changes of the maximal minors of $M$ in cyclic order. It is a generalization of the notion of sign variation defined in \cite{Machacek}, obtained by extending the linear order of the columns of $M$ to a cyclic order.

The results of this paper are summarized in the following main theorem.

\begin{theorem}\label{Thm:main}
    Let $M$ be an $m \times n$ real-valued matrix of maximal rank $m \in \{2,3\}$.
    \begin{enumerate}[(i)]
        \item If $m = 2$, then $M$ is a boundary measurement matrix for some network $N$ if and only if $\text{cvar}(M) = 2$.
        \item If $m = 3$, then $M$ is a boundary measurement matrix for some network $N$ if and only if every odd-dimensional minor is nonnegative.
    \end{enumerate}
\end{theorem}
By "odd-dimensional minor", we mean every $k \times k$ minor for odd $k$. More precisely, when $m = 3$, these are the $1 \times 1$ minors and $3 \times 3$ minors.

The paper is organized as follows. In Section \ref{sec:prelim}, we provide the necessary background and definitions. In Section \ref{sec:Column Operators}, we provide a method of constructing an arbitrary boundary measurement matrix from another "simpler" boundary measurement matrix. The main goal of this section is to explain how every boundary measurement matrix can be constructed from an identity matrix using certain column operations. We show that the properties regarding the $cvar$ function and the odd-dimensional minors in Theorem \ref{Thm:main} are preserved by these column operations. In sections \ref{sec:2xn_case} and \ref{sec:3xn_case}, we use these ideas to prove Theorem \ref{Thm:main}(i) and Theorem \ref{Thm:main}(ii), respectively. \\

\noindent \textbf{Acknowledgments.} The author is supported NSF research grant DMS \#2100791. The author is grateful to Michael Shapiro for introducing this problem and for his valuable comments towards improving the paper.

 \begin{figure}[t] 
    \centering
        \tikzset{every picture/.style={line width=0.75pt}}
\begin{tikzpicture}[x=0.75pt,y=0.75pt,yscale=-1,xscale=1]
\draw    (17.71,19.71) -- (41.45,19.71) ;
\draw [shift={(43.45,19.71)}, rotate = 180] [color={rgb, 255:red, 0; green, 0; blue, 0 }  ][line width=0.75]    (10.93,-3.29) .. controls (6.95,-1.4) and (3.31,-0.3) .. (0,0) .. controls (3.31,0.3) and (6.95,1.4) .. (10.93,3.29)   ;
\draw    (17.71,64.57) -- (41.45,64.57) ;
\draw [shift={(43.45,64.57)}, rotate = 180] [color={rgb, 255:red, 0; green, 0; blue, 0 }  ][line width=0.75]    (10.93,-3.29) .. controls (6.95,-1.4) and (3.31,-0.3) .. (0,0) .. controls (3.31,0.3) and (6.95,1.4) .. (10.93,3.29)   ;
\draw    (157.43,30.74) -- (198.91,17.38) ;
\draw [shift={(200.82,16.77)}, rotate = 162.15] [color={rgb, 255:red, 0; green, 0; blue, 0 }  ][line width=0.75]    (10.93,-3.29) .. controls (6.95,-1.4) and (3.31,-0.3) .. (0,0) .. controls (3.31,0.3) and (6.95,1.4) .. (10.93,3.29)   ;
\draw    (175.08,60.89) -- (198.82,60.89) ;
\draw [shift={(200.82,60.89)}, rotate = 180] [color={rgb, 255:red, 0; green, 0; blue, 0 }  ][line width=0.75]    (10.93,-3.29) .. controls (6.95,-1.4) and (3.31,-0.3) .. (0,0) .. controls (3.31,0.3) and (6.95,1.4) .. (10.93,3.29)   ;
\draw    (153.02,99.13) -- (198.13,109) ;
\draw [shift={(200.08,109.42)}, rotate = 192.34] [color={rgb, 255:red, 0; green, 0; blue, 0 }  ][line width=0.75]    (10.93,-3.29) .. controls (6.95,-1.4) and (3.31,-0.3) .. (0,0) .. controls (3.31,0.3) and (6.95,1.4) .. (10.93,3.29)   ;
\draw    (76.54,122.66) -- (102.87,103.26) ;
\draw [shift={(104.48,102.07)}, rotate = 143.62] [color={rgb, 255:red, 0; green, 0; blue, 0 }  ][line width=0.75]    (10.93,-3.29) .. controls (6.95,-1.4) and (3.31,-0.3) .. (0,0) .. controls (3.31,0.3) and (6.95,1.4) .. (10.93,3.29)   ;
\draw    (52.27,66.77) -- (73.46,82.49) ;
\draw [shift={(75.07,83.69)}, rotate = 216.57] [color={rgb, 255:red, 0; green, 0; blue, 0 }  ][line width=0.75]    (10.93,-3.29) .. controls (6.95,-1.4) and (3.31,-0.3) .. (0,0) .. controls (3.31,0.3) and (6.95,1.4) .. (10.93,3.29)   ;
\draw    (17.71,111.63) -- (73.18,91.72) ;
\draw [shift={(75.07,91.04)}, rotate = 160.25] [color={rgb, 255:red, 0; green, 0; blue, 0 }  ][line width=0.75]    (10.93,-3.29) .. controls (6.95,-1.4) and (3.31,-0.3) .. (0,0) .. controls (3.31,0.3) and (6.95,1.4) .. (10.93,3.29)   ;
\draw    (52.27,18.24) -- (77.27,5) ;
\draw    (83.16,88.83) -- (102.63,96.89) ;
\draw [shift={(104.48,97.66)}, rotate = 202.48] [color={rgb, 255:red, 0; green, 0; blue, 0 }  ][line width=0.75]    (10.93,-3.29) .. controls (6.95,-1.4) and (3.31,-0.3) .. (0,0) .. controls (3.31,0.3) and (6.95,1.4) .. (10.93,3.29)   ;
\draw    (111.84,97.66) -- (140.72,97.66) ;
\draw [shift={(142.72,97.66)}, rotate = 180] [color={rgb, 255:red, 0; green, 0; blue, 0 }  ][line width=0.75]    (10.93,-3.29) .. controls (6.95,-1.4) and (3.31,-0.3) .. (0,0) .. controls (3.31,0.3) and (6.95,1.4) .. (10.93,3.29)   ;
\draw    (148.61,93.25) -- (168.01,66.91) ;
\draw [shift={(169.2,65.3)}, rotate = 126.38] [color={rgb, 255:red, 0; green, 0; blue, 0 }  ][line width=0.75]    (10.93,-3.29) .. controls (6.95,-1.4) and (3.31,-0.3) .. (0,0) .. controls (3.31,0.3) and (6.95,1.4) .. (10.93,3.29)   ;
\draw    (55.21,22.65) -- (103.99,31.84) ;
\draw [shift={(105.95,32.21)}, rotate = 190.67] [color={rgb, 255:red, 0; green, 0; blue, 0 }  ][line width=0.75]    (10.93,-3.29) .. controls (6.95,-1.4) and (3.31,-0.3) .. (0,0) .. controls (3.31,0.3) and (6.95,1.4) .. (10.93,3.29)   ;
\draw    (53.01,62.36) -- (104.89,37.49) ;
\draw [shift={(106.69,36.62)}, rotate = 154.38] [color={rgb, 255:red, 0; green, 0; blue, 0 }  ][line width=0.75]    (10.93,-3.29) .. controls (6.95,-1.4) and (3.31,-0.3) .. (0,0) .. controls (3.31,0.3) and (6.95,1.4) .. (10.93,3.29)   ;
\draw    (116.98,32.21) -- (144.4,32.21) ;
\draw [shift={(146.4,32.21)}, rotate = 180] [color={rgb, 255:red, 0; green, 0; blue, 0 }  ][line width=0.75]    (10.93,-3.29) .. controls (6.95,-1.4) and (3.31,-0.3) .. (0,0) .. controls (3.31,0.3) and (6.95,1.4) .. (10.93,3.29)   ;
\draw    (154.49,36.62) -- (168.09,57.02) ;
\draw [shift={(169.2,58.68)}, rotate = 236.31] [color={rgb, 255:red, 0; green, 0; blue, 0 }  ][line width=0.75]    (10.93,-3.29) .. controls (6.95,-1.4) and (3.31,-0.3) .. (0,0) .. controls (3.31,0.3) and (6.95,1.4) .. (10.93,3.29)   ;
\draw  [dash pattern={on 0.84pt off 2.51pt}]  (14.77,5) -- (204.49,5) ;
\draw  [dash pattern={on 0.84pt off 2.51pt}]  (14.77,122.66) -- (204.49,122.66) ;
\draw    (14.77,5) -- (14.77,122.66) ;
\draw    (204.49,5) -- (204.49,122.66) ;
\draw (2.15,12.24) node [anchor=north west][inner sep=0.75pt]  [font=\normalsize]  {$1$};
\draw (206.58,9.3) node [anchor=north west][inner sep=0.75pt]  [font=\normalsize]  {$1$};
\draw (2.15,56.36) node [anchor=north west][inner sep=0.75pt]  [font=\normalsize]  {$2$};
\draw (207.32,53.42) node [anchor=north west][inner sep=0.75pt]  [font=\normalsize]  {$2$};
\draw (2.15,104.16) node [anchor=north west][inner sep=0.75pt]  [font=\normalsize]  {$3$};
\draw (206.58,101.22) node [anchor=north west][inner sep=0.75pt]  [font=\normalsize]  {$3$};
\draw (48.23,20.44) node    {$\cdot $};
\draw (48.23,63.83) node    {$\cdot $};
\draw (152.65,32.21) node    {$\cdot $};
\draw (170.67,63.1) node    {$\cdot $};
\draw (148.24,96.92) node    {$\cdot $};
\draw (78.38,88.1) node    {$\cdot $};
\draw (107.79,99.13) node    {$\cdot $};
\draw (110.73,32.21) node    {$\cdot $};
\end{tikzpicture} \\ \vspace{2ex}
    $B = \begin{bmatrix} 1 & 2 & 1 \\ 1 & 2 & 1 \\ 0 & 1 & 1 \end{bmatrix}$
    \caption{A network in a cylinder together with its boundary measurement matrix $B$. The top and bottom dotted edges of the rectangle are glued together to create the cylinder.}\label{fig:network_example}
\end{figure}
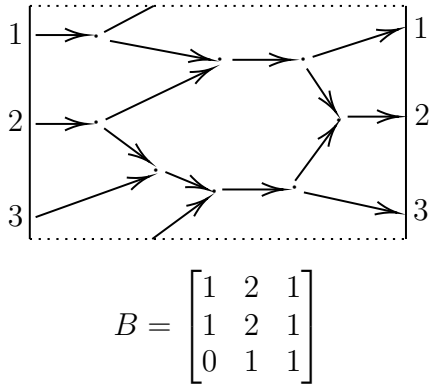

\section{Preliminaries}\label{sec:prelim}

In this section, we will define and describe networks in a cylinder. For each network, we have a collection of boundary measurements which can be represented by a nonnegative real matrix called a boundary measurement matrix.

Our main goal of the paper is to describe these matrices. In Section $\ref{sec:Column Operators}$, we will define a few basic column operators on these matrices and use these to describe our boundary matrices.

\begin{definition}
    A  \textbf{network} $N = (V,E,w)$ is a directed graph $(V,E)$ together with any positive weight function $w : E \to (0, \infty)$. We say that $N$ is a \textbf{network in a cylinder}, if $(V,E)$ can be embedded on any cylinder $S^1 \times [a,b]$ such that:
    \begin{itemize}
        \item[1.] If $v \in V \cap (S^1 \times \{a\})$, then $v$ is a source.
        \item[2.] If $v \in V \cap (S^1 \times \{b\})$, then $v$ is a sink.
        \item[3.] If $v \in V \cap (S^1 \times (a,b))$, then $v$ is neither a source nor a sink.
        \item[4.] The directed graph $(V,E)$ does not contain an oriented loop.
    \end{itemize}
\end{definition}

\begin{definition}
    For a network $N = (V,E,w)$, the \textbf{degree} or \textbf{valency} of a vertex $v \in V$, denoted $\text{deg}(v)$, is the total number of edges incident to $v$.
\end{definition}

\begin{remark}
    Since we assume that we do not have any oriented loops, we will freely assume that every edge $e \in E$ is oriented from $a$ to $b$. 
    
    Explicitly, suppose $N$ is a network in a cylinder $C = S^1 \times [a,b]$. Let $e \in E$ be an edge $v \to w$. Then we can embed $(V,E)$ into $C$ such that $\pi_2(v) < \pi_2(w)$, where $\pi_2 : C \to [a,b]$ is the projection into the second component. 

    To do this, we just need to construct an ordering $v_1 < v_2 < \cdots < v_j$ of the vertices $V = \{v_1, \cdots, v_j\}$ such that if there exists an oriented path from $v_i$ to $v_j$, then $v_i < v_j$. We can always do this if there are no oriented cycles in the underlying graph. We simply pick $v_i$ inductively, so that it is a source of the directed graph induced by removing the prior vertices $v_1, \cdots, v_{i-1}$. Then any embedding $\pi$ such that $v_i < v_j \implies \pi_2(v_i) < \pi_2(v_j)$ whenever $v_i$ and $v_j$ are interior vertices will have the arrows oriented from $a$ to $b$.
    
    To visualize this, we often draw $C$ as a rectangle $R$ whose top and bottom edges are identified as illustrated on the left in figure \ref{fig:network_cut}. The arrows we draw are pointing towards the right. In the picture, we often say that the arrows are oriented "left-to-right".
\end{remark}

\begin{definition}
    A network $N = (V,E,w)$ is called $\textbf{perfect}$ if:
    \begin{itemize}
        \item[5.] If $v \in V$ is a source or a sink, then $\text{deg}(v) = 1$.
        \item[6.] If $v \in V$ is not a source and not a sink, then $\text{deg}(v) = 3$.
    \end{itemize}
\end{definition}

\begin{definition}
    Fix a network $N = (V,E,w)$ in a cylinder. Let $p$ be a directed path from a source $i$ to a sink $j$ that travels along the arrows $\alpha_1, \alpha_2, \cdots, \alpha_n$. We define the weight $w(p)$ as the product of the weights $w(\alpha_i)$ of the arrows $\alpha_i$. That is, \[w(p) := \prod_{i=1}^k w(\alpha_i).\]
\end{definition}

\begin{definition}
    Fix a network $N = (V,E,w)$ in a cylinder. For each source $i$ and sink $j$, consider the set $P_{ij}$ consisting of all oriented paths from $i$ to $j$. We define the \textbf{boundary measurement} $M_{ij}$ as the sum of all weighted paths from $i$ to $j$:
    \begin{align*}
        M_{ij} := \sum_{p \in P_{ij}} w(p).
    \end{align*}
    Since a cylinder is an oriented surface, an orientation of the cylinder induces two orientations on each of the boundary circles. We require that both boundary circles are oriented in the same direction; both positively oriented or both negatively oriented.
    
    After orienting the boundary circles, we obtain a cyclic ordering of the sources and also a cyclic ordering of the sinks. If we fix linear orderings of these cyclic orderings, then we can represent the boundary measurements by a matrix \[M := \{M_{ij}\}.\] We call $M$ a \textbf{boundary measurement matrix} of $N$, and it is unique up to the linear orderings we choose.
\end{definition}

\begin{lemma}\label{lem:perfect_networks}
    Let $N$ be a network in a cylinder with a boundary measurement matrix $M$. Then there exists a perfect network $N'$ in a cylinder with the same boundary measurement matrix $M$.
\end{lemma}

The proof of lemma $\ref{lem:perfect_networks}$ is given for planar networks in \cite{Postnikov}. The proof is exactly the same for a cylindrical network, so it is omitted here.

\begin{remark}
    In most of the proofs in this paper, we will represent a matrix $M = [v_1 | \cdots | v_n]$ by its columns. We typically care more about the cyclic ordering of the columns instead of some fixed linear ordering. In particular, most properties of $M$ that we will describe are completely independent of the linear ordering we choose! So sometimes we will make a particular choice of linear ordering that makes the proof easier; for example, fixing an arbitrary column $v_i$ and assuming it is column $v_1$. We may do the same thing with the rows.
\end{remark}

\section{Column Operations}\label{sec:Column Operators}

The goal of this chapter is to decompose a network into an ordered sequence of smaller networks, whose boundary measurement matrices are related by a sequence of special \textbf{column operations} (definition \ref{def:col_operations}). We begin by defining these smaller networks. We will finish this section by concluding that a matrix $M$ is a boundary measurement matrix for a network if and only if it can be constructed from an identity matrix via these column operations.

\begin{definition}\label{def:subnetwork}
    Let $N = (V,E,w)$ be a network on a cylinder $C = S^1 \times [a,b]$. Fix $t \in (a,b)$. We define a new network $N_t = (V_t, E_t, w_t)$ on the cylinder $C_t = S^1 \times [a,t]$ induced by the arrows that intersect $C_t$ as follows. For each edge $e \in E$, we define a new edge $e_t$ on $C_t$ by its intersection with $C_t$. More specifically, $e_t := e \cap C_t$. Let $\ell_t$ denote the loop $S^1 \times \{t\}$. Then
    \begin{enumerate}[(i)]
        \item $V_t = (V \cap C_t) \cup (E \cap \ell_t),$ \vspace{2ex}
        \item $E_t = \{e_t : e \cap C_t \neq \emptyset\}$ \vspace{2ex}
        \item $w_t(e_t) = w(e)$.
    \end{enumerate}
\end{definition}
See figure \ref{fig:network_cut} for an explicit example.

\begin{remark} \label{rem:network_decomp}
    Since $N$ is finite, we only obtain finitely many boundary measurement matrices $M_{t_1}, \cdots, M_{t_m}$. Let $N_{t_1}, \cdots, N_{t_m}$, be the corresponding boundary measurement matrices ordered so that $t_1 < t_2 < \cdots < t_m$. Notice that $N_{t_m} = N$ and $M_{t_1}$ is a square identity matrix.
\end{remark}

\begin{figure}

\tikzset{every picture/.style={line width=0.75pt}} 

\begin{tikzpicture}[x=0.75pt,y=0.75pt,yscale=-1,xscale=1]

\draw    (20.2,20.25) -- (44.89,20.25) ;
\draw [shift={(46.89,20.25)}, rotate = 180] [color={rgb, 255:red, 0; green, 0; blue, 0 }  ][line width=0.75]    (10.93,-3.29) .. controls (6.95,-1.4) and (3.31,-0.3) .. (0,0) .. controls (3.31,0.3) and (6.95,1.4) .. (10.93,3.29)   ;
\draw    (20.2,66.76) -- (44.89,66.76) ;
\draw [shift={(46.89,66.76)}, rotate = 180] [color={rgb, 255:red, 0; green, 0; blue, 0 }  ][line width=0.75]    (10.93,-3.29) .. controls (6.95,-1.4) and (3.31,-0.3) .. (0,0) .. controls (3.31,0.3) and (6.95,1.4) .. (10.93,3.29)   ;
\draw    (165.08,31.69) -- (208.16,17.81) ;
\draw [shift={(210.07,17.2)}, rotate = 162.15] [color={rgb, 255:red, 0; green, 0; blue, 0 }  ][line width=0.75]    (10.93,-3.29) .. controls (6.95,-1.4) and (3.31,-0.3) .. (0,0) .. controls (3.31,0.3) and (6.95,1.4) .. (10.93,3.29)   ;
\draw    (183.38,62.95) -- (208.07,62.95) ;
\draw [shift={(210.07,62.95)}, rotate = 180] [color={rgb, 255:red, 0; green, 0; blue, 0 }  ][line width=0.75]    (10.93,-3.29) .. controls (6.95,-1.4) and (3.31,-0.3) .. (0,0) .. controls (3.31,0.3) and (6.95,1.4) .. (10.93,3.29)   ;
\draw    (160.5,102.6) -- (207.35,112.85) ;
\draw [shift={(209.3,113.28)}, rotate = 192.34] [color={rgb, 255:red, 0; green, 0; blue, 0 }  ][line width=0.75]    (10.93,-3.29) .. controls (6.95,-1.4) and (3.31,-0.3) .. (0,0) .. controls (3.31,0.3) and (6.95,1.4) .. (10.93,3.29)   ;
\draw    (81.2,127) -- (108.57,106.84) ;
\draw [shift={(110.18,105.65)}, rotate = 143.62] [color={rgb, 255:red, 0; green, 0; blue, 0 }  ][line width=0.75]    (10.93,-3.29) .. controls (6.95,-1.4) and (3.31,-0.3) .. (0,0) .. controls (3.31,0.3) and (6.95,1.4) .. (10.93,3.29)   ;
\draw    (56.04,69.05) -- (78.07,85.4) ;
\draw [shift={(79.68,86.59)}, rotate = 216.57] [color={rgb, 255:red, 0; green, 0; blue, 0 }  ][line width=0.75]    (10.93,-3.29) .. controls (6.95,-1.4) and (3.31,-0.3) .. (0,0) .. controls (3.31,0.3) and (6.95,1.4) .. (10.93,3.29)   ;
\draw    (20.2,115.56) -- (77.8,94.89) ;
\draw [shift={(79.68,94.21)}, rotate = 160.25] [color={rgb, 255:red, 0; green, 0; blue, 0 }  ][line width=0.75]    (10.93,-3.29) .. controls (6.95,-1.4) and (3.31,-0.3) .. (0,0) .. controls (3.31,0.3) and (6.95,1.4) .. (10.93,3.29)   ;
\draw    (56.04,18.73) -- (81.97,5) ;
\draw    (88.07,91.93) -- (108.33,100.31) ;
\draw [shift={(110.18,101.08)}, rotate = 202.48] [color={rgb, 255:red, 0; green, 0; blue, 0 }  ][line width=0.75]    (10.93,-3.29) .. controls (6.95,-1.4) and (3.31,-0.3) .. (0,0) .. controls (3.31,0.3) and (6.95,1.4) .. (10.93,3.29)   ;
\draw    (117.8,101.08) -- (147.83,101.08) ;
\draw [shift={(149.83,101.08)}, rotate = 180] [color={rgb, 255:red, 0; green, 0; blue, 0 }  ][line width=0.75]    (10.93,-3.29) .. controls (6.95,-1.4) and (3.31,-0.3) .. (0,0) .. controls (3.31,0.3) and (6.95,1.4) .. (10.93,3.29)   ;
\draw    (155.93,96.5) -- (176.09,69.14) ;
\draw [shift={(177.28,67.52)}, rotate = 126.38] [color={rgb, 255:red, 0; green, 0; blue, 0 }  ][line width=0.75]    (10.93,-3.29) .. controls (6.95,-1.4) and (3.31,-0.3) .. (0,0) .. controls (3.31,0.3) and (6.95,1.4) .. (10.93,3.29)   ;
\draw    (59.09,23.3) -- (109.74,32.84) ;
\draw [shift={(111.7,33.21)}, rotate = 190.67] [color={rgb, 255:red, 0; green, 0; blue, 0 }  ][line width=0.75]    (10.93,-3.29) .. controls (6.95,-1.4) and (3.31,-0.3) .. (0,0) .. controls (3.31,0.3) and (6.95,1.4) .. (10.93,3.29)   ;
\draw    (56.8,64.47) -- (110.66,38.65) ;
\draw [shift={(112.47,37.79)}, rotate = 154.38] [color={rgb, 255:red, 0; green, 0; blue, 0 }  ][line width=0.75]    (10.93,-3.29) .. controls (6.95,-1.4) and (3.31,-0.3) .. (0,0) .. controls (3.31,0.3) and (6.95,1.4) .. (10.93,3.29)   ;
\draw    (123.14,33.21) -- (151.64,33.21) ;
\draw [shift={(153.64,33.21)}, rotate = 180] [color={rgb, 255:red, 0; green, 0; blue, 0 }  ][line width=0.75]    (10.93,-3.29) .. controls (6.95,-1.4) and (3.31,-0.3) .. (0,0) .. controls (3.31,0.3) and (6.95,1.4) .. (10.93,3.29)   ;
\draw    (162.03,37.79) -- (176.17,59) ;
\draw [shift={(177.28,60.66)}, rotate = 236.31] [color={rgb, 255:red, 0; green, 0; blue, 0 }  ][line width=0.75]    (10.93,-3.29) .. controls (6.95,-1.4) and (3.31,-0.3) .. (0,0) .. controls (3.31,0.3) and (6.95,1.4) .. (10.93,3.29)   ;
\draw  [dash pattern={on 0.84pt off 2.51pt}]  (17.15,5) -- (213.88,5) ;
\draw  [dash pattern={on 0.84pt off 2.51pt}]  (17.15,127) -- (213.88,127) ;
\draw    (17.15,5) -- (17.15,127) ;
\draw    (213.88,5) -- (213.88,127) ;
\draw [color={rgb, 255:red, 255; green, 0; blue, 0 }  ,draw opacity=1 ]   (138.39,5) -- (138.39,127) ;
\draw    (295.38,20.25) -- (320.07,20.25) ;
\draw [shift={(322.07,20.25)}, rotate = 180] [color={rgb, 255:red, 0; green, 0; blue, 0 }  ][line width=0.75]    (10.93,-3.29) .. controls (6.95,-1.4) and (3.31,-0.3) .. (0,0) .. controls (3.31,0.3) and (6.95,1.4) .. (10.93,3.29)   ;
\draw    (295.38,66.76) -- (320.07,66.76) ;
\draw [shift={(322.07,66.76)}, rotate = 180] [color={rgb, 255:red, 0; green, 0; blue, 0 }  ][line width=0.75]    (10.93,-3.29) .. controls (6.95,-1.4) and (3.31,-0.3) .. (0,0) .. controls (3.31,0.3) and (6.95,1.4) .. (10.93,3.29)   ;
\draw    (356.38,127) -- (383.75,106.84) ;
\draw [shift={(385.36,105.65)}, rotate = 143.62] [color={rgb, 255:red, 0; green, 0; blue, 0 }  ][line width=0.75]    (10.93,-3.29) .. controls (6.95,-1.4) and (3.31,-0.3) .. (0,0) .. controls (3.31,0.3) and (6.95,1.4) .. (10.93,3.29)   ;
\draw    (331.22,69.05) -- (353.25,85.4) ;
\draw [shift={(354.86,86.59)}, rotate = 216.57] [color={rgb, 255:red, 0; green, 0; blue, 0 }  ][line width=0.75]    (10.93,-3.29) .. controls (6.95,-1.4) and (3.31,-0.3) .. (0,0) .. controls (3.31,0.3) and (6.95,1.4) .. (10.93,3.29)   ;
\draw    (295.38,115.56) -- (352.97,94.89) ;
\draw [shift={(354.86,94.21)}, rotate = 160.25] [color={rgb, 255:red, 0; green, 0; blue, 0 }  ][line width=0.75]    (10.93,-3.29) .. controls (6.95,-1.4) and (3.31,-0.3) .. (0,0) .. controls (3.31,0.3) and (6.95,1.4) .. (10.93,3.29)   ;
\draw    (331.22,18.73) -- (357.14,5) ;
\draw    (363.24,91.93) -- (383.51,100.31) ;
\draw [shift={(385.36,101.08)}, rotate = 202.48] [color={rgb, 255:red, 0; green, 0; blue, 0 }  ][line width=0.75]    (10.93,-3.29) .. controls (6.95,-1.4) and (3.31,-0.3) .. (0,0) .. controls (3.31,0.3) and (6.95,1.4) .. (10.93,3.29)   ;
\draw    (394.51,101.08) -- (407.76,101.08) ;
\draw [shift={(409.76,101.08)}, rotate = 180] [color={rgb, 255:red, 0; green, 0; blue, 0 }  ][line width=0.75]    (10.93,-3.29) .. controls (6.95,-1.4) and (3.31,-0.3) .. (0,0) .. controls (3.31,0.3) and (6.95,1.4) .. (10.93,3.29)   ;
\draw    (334.27,23.3) -- (384.92,32.84) ;
\draw [shift={(386.88,33.21)}, rotate = 190.67] [color={rgb, 255:red, 0; green, 0; blue, 0 }  ][line width=0.75]    (10.93,-3.29) .. controls (6.95,-1.4) and (3.31,-0.3) .. (0,0) .. controls (3.31,0.3) and (6.95,1.4) .. (10.93,3.29)   ;
\draw    (331.98,64.47) -- (385.84,38.65) ;
\draw [shift={(387.64,37.79)}, rotate = 154.38] [color={rgb, 255:red, 0; green, 0; blue, 0 }  ][line width=0.75]    (10.93,-3.29) .. controls (6.95,-1.4) and (3.31,-0.3) .. (0,0) .. controls (3.31,0.3) and (6.95,1.4) .. (10.93,3.29)   ;
\draw  [dash pattern={on 0.84pt off 2.51pt}]  (292.33,5) -- (413.57,5) ;
\draw  [dash pattern={on 0.84pt off 2.51pt}]  (291.57,127) -- (413.57,127) ;
\draw    (292.33,5) -- (292.33,127) ;
\draw [color={rgb, 255:red, 255; green, 0; blue, 0 }  ,draw opacity=1 ]   (413.57,5) -- (413.57,127) ;
\draw    (394.51,32.45) -- (407.76,32.45) ;
\draw [shift={(409.76,32.45)}, rotate = 180] [color={rgb, 255:red, 0; green, 0; blue, 0 }  ][line width=0.75]    (10.93,-3.29) .. controls (6.95,-1.4) and (3.31,-0.3) .. (0,0) .. controls (3.31,0.3) and (6.95,1.4) .. (10.93,3.29)   ;

\draw (4.29,12.65) node [anchor=north west][inner sep=0.75pt]  [font=\normalsize]  {$1$};
\draw (216.27,11.54) node [anchor=north west][inner sep=0.75pt]  [font=\normalsize]  {$1$};
\draw (4.29,58.4) node [anchor=north west][inner sep=0.75pt]  [font=\normalsize]  {$2$};
\draw (216.27,56.52) node [anchor=north west][inner sep=0.75pt]  [font=\normalsize]  {$2$};
\draw (4.29,107.96) node [anchor=north west][inner sep=0.75pt]  [font=\normalsize]  {$3$};
\draw (216.27,105.88) node [anchor=north west][inner sep=0.75pt]  [font=\normalsize]  {$3$};
\draw (51.85,21.01) node    {$\cdot $};
\draw (51.85,66) node    {$\cdot $};
\draw (160.12,33.21) node    {$\cdot $};
\draw (178.8,65.24) node    {$\cdot $};
\draw (155.55,100.31) node    {$\cdot $};
\draw (83.11,91.16) node    {$\cdot $};
\draw (113.61,102.6) node    {$\cdot $};
\draw (116.66,33.21) node    {$\cdot $};
\draw (280.44,10.71) node [anchor=north west][inner sep=0.75pt]  [font=\normalsize]  {$1$};
\draw (415.19,23.33) node [anchor=north west][inner sep=0.75pt]  [font=\normalsize]  {$1$};
\draw (279.47,57.43) node [anchor=north west][inner sep=0.75pt]  [font=\normalsize]  {$2$};
\draw (416.72,91.95) node [anchor=north west][inner sep=0.75pt]  [font=\normalsize]  {$2$};
\draw (277.53,106.03) node [anchor=north west][inner sep=0.75pt]  [font=\normalsize]  {$3$};
\draw (327.02,21.01) node    {$\cdot $};
\draw (327.02,66) node    {$\cdot $};
\draw (358.29,91.16) node    {$\cdot $};
\draw (388.79,102.6) node    {$\cdot $};
\draw (391.84,33.21) node    {$\cdot $};

\end{tikzpicture}
 \\ \vspace{3ex}
    $\qquad\begin{bmatrix} 1 & 2 & 1 \\ 1 & 2 & 1 \\ 0 & 1 & 1 \end{bmatrix}  \qquad\qquad\qquad\qquad\qquad \begin{bmatrix} 1 & 1 \\ 1 & 1 \\ 0 & 1 \end{bmatrix}$

    \caption{A network $N$ with boundary measurement matrix $M$ (left) and the network $N_t$ with boundary measurement matrix $M_t$ (right). } \label{fig:network_cut}
\end{figure}
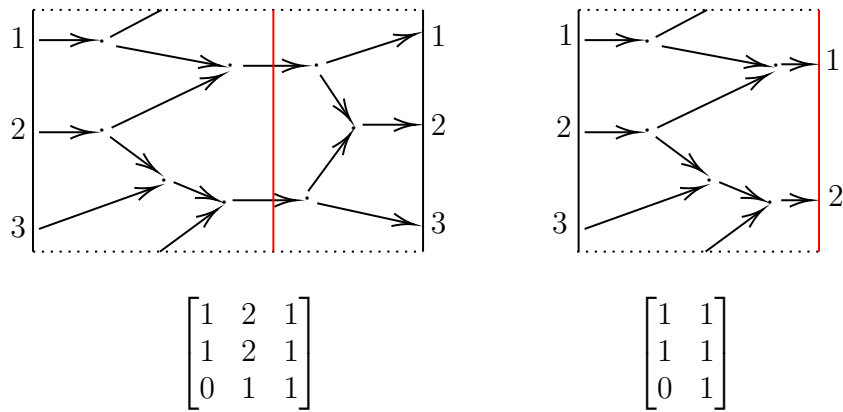

\begin{definition}\label{def:col_operations}
    Let $M = [v_1 | \cdots | v_n]$ be any $m \times n$ matrix. We consider the following four elementary column operations:
    \begin{align*}
        &\text{\textbf{1. Join}:} &&[v_1 | \cdots |v_i | v_{i+1} | \cdots | v_n] &&\mapsto &&[v_1 | \cdots |v_i + v_{i+1} | \cdots | v_n], \\
        &\text{\textbf{2. Double}:} &&[v_1 | \cdots |v_i | \cdots | v_n] &&\mapsto &&[v_1 | \cdots |v_i | v_i | \cdots | v_n], \\
        &\text{\textbf{3. Cyclic Shift}:} &&[v_1 | v_2 | \cdots | v_{n-1} | v_n] &&\mapsto &&[v_n | v_1 | v_2 | \cdots | v_{n-1}], \\
        &\text{\textbf{4. Re-scale}:} &&[v_1 | \cdots |v_i | \cdots | v_n] &&\mapsto &&[v_1 | \cdots |a v_i | \cdots | v_n], \qquad (a \geq 0).
    \end{align*}
    We say that $M$ can be \textbf{constructed} from $M'$, if $M$ can be obtained from $M'$ by applying some sequence of these column operations.
\end{definition}

\begin{remark}
    The property is transitive. That is to say if $M$ can be constructed from $M'$ and $M'$ can be constructed from $M''$, then $M$ can be constructed from $M''$.
\end{remark}

    We can interpret these operations in terms of local changes in the networks. These are illustrated in figure \ref{fig:network_operations}. This is quite constructive. We immediately obtain the following proposition.

\begin{proposition}\label{prop:construct_network}
    Let $M$ be a nonnegative $m \times n$ real matrix. Then $M$ is a boundary measurement matrix for some network $N$ if and only if $M$ can be constructed from the $m \times m$ identity matrix $I_m$. 
\end{proposition}

\begin{proof}
    If $M$ is a boundary measurement matrix for some network $N$, lemma \ref{lem:perfect_networks} tells us that we may assume $N$ is a perfect network. Let $I_m = M_{t_1}, \cdots, M_{t_j} = M$ be the networks defined via \ref{rem:network_decomp}. The difference between two consecutive networks $N_{t_i}$ and $N_{t_{i+1}}$ are minimal by construction; the only possible changes for a perfect network are illustrated in figure \ref{fig:network_operations}. By transitivity, this means that $M_{t_j} = M$ can be constructed from $M_{t_1} = I_m$.

    Conversely, if $M$ can be constructed from the identity matrix $I_m$, then figure \ref{fig:network_operations} illustrates how we construct $M$ from our sequence of column operations. 
\end{proof}   

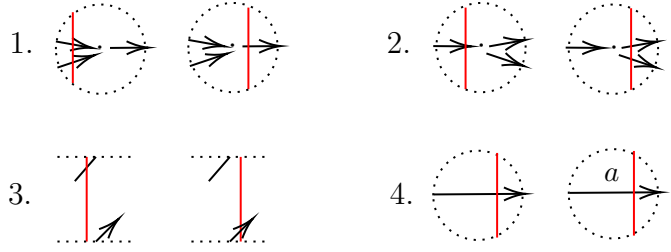
\begin{figure} 
    \tikzset{every picture/.style={line width=0.75pt}} 

\tikzset{every picture/.style={line width=0.75pt}} 

\begin{tikzpicture}[x=0.75pt,y=0.75pt,yscale=-1,xscale=1]

\draw    (46.86,34.42) -- (63.3,37.34) ;
\draw [shift={(65.27,37.69)}, rotate = 190.08] [color={rgb, 255:red, 0; green, 0; blue, 0 }  ][line width=0.75]    (10.93,-3.29) .. controls (6.95,-1.4) and (3.31,-0.3) .. (0,0) .. controls (3.31,0.3) and (6.95,1.4) .. (10.93,3.29)   ;
\draw    (47.63,47.62) -- (64.74,41.88) ;
\draw [shift={(66.64,41.25)}, rotate = 161.44] [color={rgb, 255:red, 0; green, 0; blue, 0 }  ][line width=0.75]    (10.93,-3.29) .. controls (6.95,-1.4) and (3.31,-0.3) .. (0,0) .. controls (3.31,0.3) and (6.95,1.4) .. (10.93,3.29)   ;
\draw    (74.15,37.69) -- (91.46,37.54) ;
\draw [shift={(93.46,37.53)}, rotate = 179.51] [color={rgb, 255:red, 0; green, 0; blue, 0 }  ][line width=0.75]    (10.93,-3.29) .. controls (6.95,-1.4) and (3.31,-0.3) .. (0,0) .. controls (3.31,0.3) and (6.95,1.4) .. (10.93,3.29)   ;
\draw [color={rgb, 255:red, 255; green, 0; blue, 0 }  ,draw opacity=1 ]   (55.4,19.66) -- (55.4,55.39) ;
\draw  [dash pattern={on 0.84pt off 2.51pt}] (46.86,38.3) .. controls (46.86,25.86) and (56.94,15.78) .. (69.38,15.78) .. controls (81.82,15.78) and (91.91,25.86) .. (91.91,38.3) .. controls (91.91,50.74) and (81.82,60.83) .. (69.38,60.83) .. controls (56.94,60.83) and (46.86,50.74) .. (46.86,38.3) -- cycle ;
\draw    (113.66,33.64) -- (130.1,36.56) ;
\draw [shift={(132.07,36.91)}, rotate = 190.08] [color={rgb, 255:red, 0; green, 0; blue, 0 }  ][line width=0.75]    (10.93,-3.29) .. controls (6.95,-1.4) and (3.31,-0.3) .. (0,0) .. controls (3.31,0.3) and (6.95,1.4) .. (10.93,3.29)   ;
\draw    (114.44,46.85) -- (131.54,41.1) ;
\draw [shift={(133.44,40.47)}, rotate = 161.44] [color={rgb, 255:red, 0; green, 0; blue, 0 }  ][line width=0.75]    (10.93,-3.29) .. controls (6.95,-1.4) and (3.31,-0.3) .. (0,0) .. controls (3.31,0.3) and (6.95,1.4) .. (10.93,3.29)   ;
\draw    (140.96,36.91) -- (158.27,36.77) ;
\draw [shift={(160.27,36.75)}, rotate = 179.51] [color={rgb, 255:red, 0; green, 0; blue, 0 }  ][line width=0.75]    (10.93,-3.29) .. controls (6.95,-1.4) and (3.31,-0.3) .. (0,0) .. controls (3.31,0.3) and (6.95,1.4) .. (10.93,3.29)   ;
\draw [color={rgb, 255:red, 255; green, 0; blue, 0 }  ,draw opacity=1 ]   (143.95,17.33) -- (143.95,58.5) ;
\draw  [dash pattern={on 0.84pt off 2.51pt}] (113.66,37.53) .. controls (113.66,25.09) and (123.74,15) .. (136.19,15) .. controls (148.63,15) and (158.71,25.09) .. (158.71,37.53) .. controls (158.71,49.97) and (148.63,60.05) .. (136.19,60.05) .. controls (123.74,60.05) and (113.66,49.97) .. (113.66,37.53) -- cycle ;
\draw    (348.61,34.77) -- (332.16,37.69) ;
\draw [shift={(350.58,34.42)}, rotate = 169.92] [color={rgb, 255:red, 0; green, 0; blue, 0 }  ][line width=0.75]    (10.93,-3.29) .. controls (6.95,-1.4) and (3.31,-0.3) .. (0,0) .. controls (3.31,0.3) and (6.95,1.4) .. (10.93,3.29)   ;
\draw    (347.9,46.99) -- (330.79,41.25) ;
\draw [shift={(349.8,47.62)}, rotate = 198.56] [color={rgb, 255:red, 0; green, 0; blue, 0 }  ][line width=0.75]    (10.93,-3.29) .. controls (6.95,-1.4) and (3.31,-0.3) .. (0,0) .. controls (3.31,0.3) and (6.95,1.4) .. (10.93,3.29)   ;
\draw    (321.28,37.67) -- (303.97,37.53) ;
\draw [shift={(323.28,37.69)}, rotate = 180.49] [color={rgb, 255:red, 0; green, 0; blue, 0 }  ][line width=0.75]    (10.93,-3.29) .. controls (6.95,-1.4) and (3.31,-0.3) .. (0,0) .. controls (3.31,0.3) and (6.95,1.4) .. (10.93,3.29)   ;
\draw [color={rgb, 255:red, 255; green, 0; blue, 0 }  ,draw opacity=1 ]   (336.96,17.64) -- (336.96,58.81) ;
\draw  [dash pattern={on 0.84pt off 2.51pt}] (350.58,38.3) .. controls (350.58,25.86) and (340.49,15.78) .. (328.05,15.78) .. controls (315.61,15.78) and (305.52,25.86) .. (305.52,38.3) .. controls (305.52,50.74) and (315.61,60.83) .. (328.05,60.83) .. controls (340.49,60.83) and (350.58,50.74) .. (350.58,38.3) -- cycle ;
\draw    (281.8,33.99) -- (265.36,36.91) ;
\draw [shift={(283.77,33.64)}, rotate = 169.92] [color={rgb, 255:red, 0; green, 0; blue, 0 }  ][line width=0.75]    (10.93,-3.29) .. controls (6.95,-1.4) and (3.31,-0.3) .. (0,0) .. controls (3.31,0.3) and (6.95,1.4) .. (10.93,3.29)   ;
\draw    (281.1,46.21) -- (263.99,40.47) ;
\draw [shift={(283,46.85)}, rotate = 198.56] [color={rgb, 255:red, 0; green, 0; blue, 0 }  ][line width=0.75]    (10.93,-3.29) .. controls (6.95,-1.4) and (3.31,-0.3) .. (0,0) .. controls (3.31,0.3) and (6.95,1.4) .. (10.93,3.29)   ;
\draw    (254.48,36.9) -- (237.17,36.75) ;
\draw [shift={(256.48,36.91)}, rotate = 180.49] [color={rgb, 255:red, 0; green, 0; blue, 0 }  ][line width=0.75]    (10.93,-3.29) .. controls (6.95,-1.4) and (3.31,-0.3) .. (0,0) .. controls (3.31,0.3) and (6.95,1.4) .. (10.93,3.29)   ;
\draw [color={rgb, 255:red, 255; green, 0; blue, 0 }  ,draw opacity=1 ]   (253.48,17.33) -- (253.48,58.5) ;
\draw  [dash pattern={on 0.84pt off 2.51pt}] (283.77,37.53) .. controls (283.77,25.09) and (273.69,15) .. (261.25,15) .. controls (248.81,15) and (238.72,25.09) .. (238.72,37.53) .. controls (238.72,49.97) and (248.81,60.05) .. (261.25,60.05) .. controls (273.69,60.05) and (283.77,49.97) .. (283.77,37.53) -- cycle ;
\draw    (56.41,104.94) -- (67.05,92.68) ;
\draw [color={rgb, 255:red, 255; green, 0; blue, 0 }  ,draw opacity=1 ]   (62.39,92.68) -- (62.39,135.4) ;
\draw  [dash pattern={on 0.84pt off 2.51pt}]  (47.63,92.68) -- (86.47,92.68) ;
\draw  [dash pattern={on 0.84pt off 2.51pt}]  (47.63,135.4) -- (86.47,135.4) ;
\draw    (67.05,135.4) -- (77.29,125.16) ;
\draw [shift={(78.7,123.75)}, rotate = 135] [color={rgb, 255:red, 0; green, 0; blue, 0 }  ][line width=0.75]    (10.93,-3.29) .. controls (6.95,-1.4) and (3.31,-0.3) .. (0,0) .. controls (3.31,0.3) and (6.95,1.4) .. (10.93,3.29)   ;
\draw    (123.99,104.94) -- (134.63,92.68) ;
\draw [color={rgb, 255:red, 255; green, 0; blue, 0 }  ,draw opacity=1 ]   (140.07,92.68) -- (140.07,135.4) ;
\draw  [dash pattern={on 0.84pt off 2.51pt}]  (115.21,92.68) -- (154.05,92.68) ;
\draw  [dash pattern={on 0.84pt off 2.51pt}]  (115.21,135.4) -- (154.05,135.4) ;
\draw    (134.63,135.4) -- (144.87,125.16) ;
\draw [shift={(146.28,123.75)}, rotate = 135] [color={rgb, 255:red, 0; green, 0; blue, 0 }  ][line width=0.75]    (10.93,-3.29) .. controls (6.95,-1.4) and (3.31,-0.3) .. (0,0) .. controls (3.31,0.3) and (6.95,1.4) .. (10.93,3.29)   ;
\draw    (237.53,111.63) -- (281.77,111.33) ;
\draw [shift={(283.77,111.32)}, rotate = 179.62] [color={rgb, 255:red, 0; green, 0; blue, 0 }  ][line width=0.75]    (10.93,-3.29) .. controls (6.95,-1.4) and (3.31,-0.3) .. (0,0) .. controls (3.31,0.3) and (6.95,1.4) .. (10.93,3.29)   ;
\draw  [dash pattern={on 0.84pt off 2.51pt}] (237.17,112.1) .. controls (237.17,99.66) and (247.25,89.57) .. (259.69,89.57) .. controls (272.13,89.57) and (282.22,99.66) .. (282.22,112.1) .. controls (282.22,124.54) and (272.13,134.62) .. (259.69,134.62) .. controls (247.25,134.62) and (237.17,124.54) .. (237.17,112.1) -- cycle ;
\draw    (305.89,110.85) -- (350.13,110.56) ;
\draw [shift={(352.13,110.54)}, rotate = 179.62] [color={rgb, 255:red, 0; green, 0; blue, 0 }  ][line width=0.75]    (10.93,-3.29) .. controls (6.95,-1.4) and (3.31,-0.3) .. (0,0) .. controls (3.31,0.3) and (6.95,1.4) .. (10.93,3.29)   ;
\draw  [dash pattern={on 0.84pt off 2.51pt}] (305.52,111.32) .. controls (305.52,98.88) and (315.61,88.79) .. (328.05,88.79) .. controls (340.49,88.79) and (350.58,98.88) .. (350.58,111.32) .. controls (350.58,123.76) and (340.49,133.85) .. (328.05,133.85) .. controls (315.61,133.85) and (305.52,123.76) .. (305.52,111.32) -- cycle ;
\draw [color={rgb, 255:red, 255; green, 0; blue, 0 }  ,draw opacity=1 ]   (269.39,90.68) -- (269.39,133.4) ;
\draw [color={rgb, 255:red, 255; green, 0; blue, 0 }  ,draw opacity=1 ]   (338.39,89.68) -- (338.39,132.4) ;

\draw (69.12,37.69) node    {$\cdot $};
\draw (135.92,36.91) node    {$\cdot $};
\draw (21.99,28.48) node [anchor=north west][inner sep=0.75pt]    {$1.$};
\draw (212.3,28.48) node [anchor=north west][inner sep=0.75pt]    {$2.$};
\draw (328.31,37.69) node  [xscale=-1]  {$\cdot $};
\draw (261.51,36.91) node  [xscale=-1]  {$\cdot $};
\draw (21.21,104.61) node [anchor=north west][inner sep=0.75pt]    {$3.$};
\draw (213.85,104.61) node [anchor=north west][inner sep=0.75pt]    {$4.$};
\draw (321.83,97.06) node [anchor=north west][inner sep=0.75pt]  [font=\small]  {$a$};

\end{tikzpicture}
    \caption{An illustration of the column operations used to construct a network $N$.} \label{fig:network_operations}
\end{figure}

\begin{lemma}
    Let $M$, $M'$, $M''$ be any three matrices, such that $M$ can be constructed from $M'$ using only the re-scale and cyclic shift operators. Then we have the following properties:
    \begin{itemize}
        \item[(a)] $M'$ can be constructed from $M$.
        \item[(b)] $M$ can be constructed from $M''$ if and only if $M'$ can be constructed from $M''$.
    \end{itemize}
\end{lemma}

\begin{proof}
    Part (a) is immediately apparent, because the re-scale and cyclic shift operators are invertible in the ring generated by these four elementary column operators. Part (b) follows immediately.
\end{proof}

\begin{corollary}
    Let $N = (V,E,w)$ be a network on a cylinder. Let $M$ and $M'$ be any two different representations of the boundary measurements $\{M_{ij}\}$. Then for any matrix $M''$, we see that $M$ can be constructed from $M''$ if and only if $M'$ can be constructed from $M''$. In other words, whether or not $M$ can be constructed from $M''$ is independent of the choice of representation we choose for the boundary measurements $\{M_{ij}\}$.
\end{corollary}

\begin{remark}
    The fact that some boundary measurement matrix $M$ can be constructed from an identity matrix $I_m$ is independent of the linear orderings we chose in our definition of $M$; it depends only on the circular ordering of its columns.
\end{remark}

\section{2xn Case}\label{sec:2xn_case}

In this section, we describe the $2 \times n$ boundary measurement matrices. We begin by defining a measurement called the cyclic sign variation. Its main goal is to describe the number of columns of a boundary matrix $M$, and to describe the set of boundary $2 \times n$ boundary measurement matrices later in theorem $\ref{Thm:main_2xn}$.

We begin with the definition of the sign variation of a vector $v \in \R^n$, as defined in \cite{Machacek}.

\begin{definition}
    Let $v = (v_1, v_2, \cdots, v_n)$ be a vector in $\R^n$. We define the \textbf{sign variation} of $v$, denoted by $\text{var}(v)$, to be the number of sign changes in the entries of $v$ after removing all zeros. More precisely, let $w = (w_1, \cdots, w_k)$ be the tuple obtained from $v$ by removing all zeros. Then
    \begin{align*}
        \text{var}(v) :=\# \{i < k \mid w_i w_{i+1} < 0\}.
    \end{align*}
\end{definition}

\begin{definition}\label{def:cvar0}
    Let $v = (v_1, v_2, \cdots, v_n)$ be a vector in $\R^n$, whose indices are considered modulo $n$. We define the \textbf{cyclic sign variation} of $v$, denoted by $\text{cvar}(v)$, as the number of sign changes of $v$ in this cyclic ordering. More precisely, if $v$ is nonzero and $v_k$ is the first nonzero entry of $v$, then we can define $\text{cvar}$ in terms of $\text{var}$ by \[\text{cvar}(v) = \text{cvar}(0, 0, \cdots, 0, v_k, \cdots, v_n) = \text{var}(v_k, \cdots, v_n, v_k).\]
    Otherwise, if $v$ is the zero vector, we just say $\text{cvar}(v) = 0$.
\end{definition}

\begin{remark}\label{rem:cvar_even}
    We see that $cvar(v)$ is always even; every time the sign changes, it must change back. This is due to the fact that the first and last entry of the tuple are the same.
\end{remark}

We generalize the notion of cyclic sign variation to an arbitrary $m \times n$ matrix by constructing a $n$-tuple consisting of the maximal minors given by consecutive columns in cyclic order.

\begin{definition}\label{def:cvar}
    Let $M = [v_1 | \cdots | v_n]$ be a nonnegative $m \times n$ real-valued matrix. Let $\Delta_{i} := \det{[v_i \mid v_{i+1} \mid \cdots \mid v_{i+m-1}]}$, where the indices are considered modulo $n$. We define the \textbf{cyclic sign variation} of $M$, denoted by $\text{cvar}(M)$, to be 
    \[\text{cvar}(M) := \text{cvar}(\Delta_1, \Delta_2, \cdots, \Delta_n).\]
\end{definition}

\begin{remark}\label{rem:cvar_iden}
    It follows from our column operations that if $M$ can be constructed from $M'$, then $\text{cvar}(M') \geq \text{cvar}(M)$. As an immediate consequence, if $M$ can be constructed from $I_2$, then $\text{cvar}(M) \leq \text{cvar}(I_2) = 2$.
\end{remark}

\begin{remark}\label{rem:cvar_counting}
    Suppose that for all $i = 1, \cdots, n$, the signs of these determinants are also alternating, i.e.
    \begin{align*}
        \Delta_i \Delta_{i+1} < 0.
    \end{align*}
    Then the $\text{cvar}$ function counts the number of columns of $M$.
\end{remark}

We will think of column vectors as elements of $\mathbb{R}^2$ in the first quadrant. Recall the principal argument $\text{Arg}(v) \in [0, \frac{\pi}{2}]$ of a column vector $v$ is the angle between it and the positive $x$-axis in counter-clockwise direction.

\begin{lemma}\label{lem:2xm_reduction}
    Let $M = [v_1 | \cdots | v_n]$ be a nonnegative $2 \times n$ real-valued matrix. Assume that none of the columns of $M$ is the zero vector. Fix three consecutive columns $v_{i-1}, v_i, v_{i+1}$ of $M$. If either
    \begin{align*}
         \text{Arg}(v_{i-1}) \leq \text{Arg}(v_{i}) \leq \text{Arg}(v_{i+1}) \qquad \text{or} \qquad \text{Arg}(v_{i-1}) \geq \text{Arg}(v_{i}) \geq \text{Arg}(v_{i+1}),
    \end{align*}
    then for some $a,b \geq 0$, we have
    \begin{align*}
        v_i = av_{i-1} + bv_{i+1}.
    \end{align*}
    In particular, $M$ can be constructed from $M' := [v_1 | \cdots | v_{i-1} | v_{i+1} | \cdots | v_n]$ with $\text{cvar}(M) = \text{cvar}(M')$, obtained by removing the column $v_i$ from $M$.
\end{lemma}

\begin{proof}
    Without loss of generality, we assume the first case $\text{Arg}(v_{i-1}) \leq \text{Arg}(v_{i}) \leq \text{Arg}(v_{i+1})$. In the cases where $\text{Arg}(v_{i-1}) = \text{Arg}(v_{i})$ or $\text{Arg}(v_{i}) = \text{Arg}(v_{i+1})$, we see that $v_i = a v_{i-1}$ or $v_i = a v_{i+1}$. It follows that $a > 0$, since all three column vectors are nonnegative.

    Next, assume $\text{Arg}(v_{i-1}) < \text{Arg}(v_{i}) < \text{Arg}(v_{i+1})$. In particular, $v_{i-1}$ and $v_{i+1}$ span the column space of $M$, so $v_i = av_{i-1} + bv_{i+1}$ for some $a, b \in \mathbb{R}$. Since the arguments are increasing and the columns are nonnegative, this implies all of the determinants $\text{det}(v_{i-1}, v_i)$, $\text{det}(v_{i-1}, v_{i+1})$, $\text{det}(v_{i}, v_{v_i+1})$ are positive. Therefore
    \begin{align*}
        &0 < \text{det}(v_{i-1}, v_i) = b \cdot \text{det}(v_{i-1}, v_{i+1}), \\ &0 < \text{det}(v_{i}, v_{i+1}) = a \cdot \text{det}(v_{i-1}, v_{i+1}). 
    \end{align*}
    By cancellation, we see that both $a,b > 0$.
\end{proof}

\begin{corollary}\label{cor:2x2_const}
    Let $M = [v_1 | \cdots | v_n]$ be a nonnegative $2 \times n$ real-valued matrix. If $\text{cvar}(M) \leq 2$, then $M$ can be constructed from some nonnegative $2 \times 2$ real-valued matrix.
\end{corollary}

\begin{proof}
    We know that if $v_i \in \text{span}(v_{i-1},v_{i+1})$, then $M$ can be constructed from the matrix $[v_1 | \cdots | v_{i-1} | v_{i+1} | \cdots | v_n]$ obtained by removing the column $v_i$. Let $M'$ be the matrix obtained by removing a maximal number of columns this way. Then $M$ can be constructed from $M'$. We just need to show that $M'$ has at most two columns.
    
    The contrapositive of the previous lemma implies that if we cannot remove any more columns, then the signs $\text{Arg}(v_{i+1}) - \text{Arg}(v_{i})$ form an alternating sequence. It is easy to show that $\text{Arg}(v_{i+1}) - \text{Arg}(v_{i})$ and $\Delta_{i+1}$ have the same sign. Since $\text{cvar}(M') \leq \text{cvar}(M) \leq 2$, this tells us that $M'$ has at most $2$ columns. 
\end{proof}

\begin{theorem}\label{Thm:main_2xn}
    Let $M = [v_1 | \cdots | v_n]$ be a nonnegative $2\times n$ real-valued matrix. Suppose $\text{rank}(M) = 2$. Then $\text{cvar}(M) = 2$ if and only if $M$ can be constructed from the $2 \times 2$ identity matrix $I_2$.
\end{theorem}

\begin{proof}
    If $M$ can be constructed from $I_2$, then $\text{cvar}(M) \leq 2$ follows immediately by remark $\ref{rem:cvar_iden}$. So one direction is immediately clear. So we will assume $\text{cvar}(M) = 2$ and prove that $M$ can be constructed from $I_2$.
    
    Suppose $\text{cvar}(M) = 2$. By corollary $\ref{cor:2x2_const}$, we know that $M$ can be constructed from a nonnegative $2 \times 2$ matrix $M'$. We can easily construct $M'$ from $I_2$:
    \begin{equation*}
        \begin{bmatrix} 1 & 0 \\ 0 & 1 \end{bmatrix} \mapsto \begin{bmatrix} 1 & 0 & 0 & 1 \\ 0 & 1 & 1 & 0 \end{bmatrix} \mapsto \begin{bmatrix} (M')_{11} & 0 & 0 & (M')_{12} \\ 0 & (M')_{21} & (M')_{22} & 0 \end{bmatrix} \mapsto \begin{bmatrix} (M')_{11} & (M')_{12} \\(M')_{21} & (M')_{22} \end{bmatrix} = M'.
    \end{equation*}
    
    Hence $M$ can be constructed from some $2 \times 2$ nonnegative matrix $M'$, which can be constructed from the $2 \times 2$ identity matrix $I_2$.
\end{proof}

\section{3xn Case}\label{sec:3xn_case}

In this section, we classify the $3 \times n$ boundary matrices. We begin with a higher dimensional analogue to lemma $\ref{lem:2xm_reduction}$. Instead of this property that $\text{cvar}(M) \leq 2$ in the previous case, we will look at matrices whose $3 \times 3$ minors are also nonnegative.

\begin{lemma}\label{lem:col_rem}
    Let $M = [v_1 | \cdots | v_n]$ be a nonnegative $m \times n$ real-valued matrix. Assume every $3 \times 3$ minor of $M$ is nonnegative and $\text{rank}(M) \geq 3$. Fix three consecutive columns $v_{i-1}, v_i, v_{i+1}$ of $M$. 
    
    If $v_{i-1}, v_i, v_{i+1}$ are linearly dependent, then for some $a,b \geq 0$, we have \[v_i = av_{i-1} + bv_{i+1}.\] In particular, $M$ can be constructed from the matrix $[v_1 | \cdots | v_{i-1} | v_{i+1} | \cdots | v_n]$ obtained by removing the column $v_i$ from $M$.
\end{lemma}

\begin{proof}
    Suppose $v_{i-1}$, $v_i$, $v_{i+1}$ are linearly dependent. If $v_{i-1}, v_i$ are linearly dependent, then $v_{i-1},v_i$ are scalar multiplies of one another, so $M$ is obtained from $M'$ by duplicating and rescaling $v_i$. So there is nothing to show. Similarly, if $v_i$, $v_{i+1}$ are linearly dependent, then there is nothing to show. 
    
    So we proceed by assuming these pairs are linearly independent. In fact, because $v_{i-1}$, $v_i$, $v_{i+1}$ are linearly dependent, these pairs span the same $2$-dimensional vectors spaces, i.e., \[\text{span}(v_{i-1}, v_i) = \text{span}(v_i, v_{i+1}).\] Therefore, there exist some $a, b \in \mathbb{R}$ such that \[v_{i+1} = av_{i-1} + bv_i.\] This expression is slightly different from our statement; nevertheless, it suffices to show that $a < 0$ and $b \geq 0$.
    
    Now by assumption, $\text{rank}(M) \geq 3$. So there exists another column vector $v_k$ that completes this $2$-dimensional space to a $3$-dimensional space. Since $v_{i}$, $v_{i+1}$, $v_{k}$ are linearly independent, there exists a $3 \times n$ sub-matrix $\Tilde{M} = [\Tilde{v}_1 | \cdots | \Tilde{v}_n]$ of $M$ such that
    \begin{align*}
        0 < \det([\Tilde{v}_{i} | \Tilde{v}_{i+1} | \Tilde{v}_{k}]) = -a \det([\Tilde{v}_{i-1} | \Tilde{v}_{i} | \Tilde{v}_{k}]).
    \end{align*}
    So we conclude $a < 0$. Since the columns $v_{i-1}, v_i, v_{i+1}$ are nonnegative, we cannot have $b \leq 0$. So $b > 0$ and we are done.
\end{proof}

\begin{lemma}\label{lem:col_indep}
    Let $M = [v_1 | \cdots | v_n]$ be a nonnegative $m \times n$ real-valued matrix. Assume every $3 \times 3$ minor of $M$ is nonnegative and $\text{rank}(M) \geq 3$. 
    
    If every three consecutive columns of $M$ (in cyclic order) are linearly independent, then every three distinct columns of $M$ are linearly independent.
\end{lemma}

\begin{proof}
    Assume every three consecutive columns of $M$ are linearly independent. We will show that if we remove some arbitrary column $v_i$ from $M$, then the resulting matrix $M'$ also satisfies this property. If we can prove this, then we can then remove all but three columns from $M$, concluding that these three column vectors are linearly independent. We proceed by induction on $m \geq 3$; our base case is clear because $\text{rank}(M) \geq 3$.

    So assume $m \geq 4$. For an arbitrary column $v_i$, we need to show two things; first that $v_{i-2}, v_{i-1}, v_{i+1}$ are linearly independent, and second that $v_{i-1}, v_{i+1}, v_{i+2}$ are linearly independent. The proofs are similar, so we prove the first.
    
    We proceed by contradiction, so assume $v_{i-2}, v_{i-1}, v_{i+1}$ are linearly dependent. Our assumptions on $M$ tell us that $v_{i-2}, v_{i-1}$ are linearly independent, so we can find $a, b \in \mathbb{R}$ such that \[v_{i+1} = av_{i-2} + bv_{i-1}.\] Now since every three consecutive columns of $M$ are linearly independent, we can find a $3 \times n$ sub-matrix $[\Tilde{v}_1 | \cdots | \Tilde{v}_n]$ of $M$ and also a $3 \times n$ sub-matrix $[\widehat{v}_1 | \cdots | \widehat{v}_n]$ of $M$ such that $\det([\Tilde{v}_{i-3} | \Tilde{v}_{i-2} | \Tilde{v}_{i-1}]) > 0$ and $\det([\widehat{v}_{i-2} | \widehat{v}_{i-1} | \widehat{v}_{i}]) > 0$. Then we have inequalities
    \begin{align*}
        0 &\leq \det([\Tilde{v}_{i-3} | \Tilde{v}_{i-2} | \Tilde{v}_{i+1}]) = b \det([\Tilde{v}_{i-3} | \Tilde{v}_{i-2} | \Tilde{v}_{i-1}]), \\
        0 &<  \det([\widehat{v}_{i-2} | \widehat{v}_{i} | \widehat{v}_{i+1}]) = -b \det([\widehat{v}_{i-2} | \widehat{v}_{i-1} | \widehat{v}_{i}]).
    \end{align*}
    The first inequality implies $b \geq 0$, and the second implies $b < 0$. Hence a contradiction, and we are done.
\end{proof}

\begin{corollary}\label{cor:3x3_pos}
    Let $M = [v_1 | \cdots | v_n]$ be a nonnegative $3 \times n$ real-valued matrix. Assume every $3 \times 3$ minor of $M$ is nonnegative and $\text{rank}(M) = 3$. Then $M$ can be constructed from a nonnegative $3 \times n$ real-valued matrix $M'$ such that every $3 \times 3$ minor of $M$ is strictly positive and $\text{rank}(M') = 3$.    
\end{corollary}

\begin{proof}
    This is an immediate consequence of applying lemma $\ref{lem:col_rem}$ until every three consecutive columns are linearly independent. Then when $m = 3$, lemma $\ref{lem:col_indep}$ tells us that every $3 \times 3$ minor of this resulting matrix must be positive.
\end{proof}

Next, we define a partial ordering on the columns of $M$, which will we used in the proceeding lemma.

\begin{definition}
    Let $M = [v_1 | \cdots | v_n]$ be a nonnegative $m \times n$ real-valued matrix. For any two columns $v_i$, $v_j$, we define a partial ordering $v_i \preceq v_j$ if and only if $(v_i)_k = 0 \implies (v_j)_k = 0$ for every row $k$ of $M$. If $v_i \preceq v_j$ or $v_j \preceq v_i$, then we say that $v_i$ and $v_j$ are comparable.
\end{definition}

\begin{lemma}\label{lem:col_subtract}
    Let $M = [v_1 | \cdots | v_n]$ be a nonnegative $3 \times n$ real-valued matrix. Assume every $3 \times 3$ maximal minor of $M$ is strictly positive and $\text{rank}(M) = 3$. Fix a column vector $v_i$ of $M$ and one of its neighbors $v_j \in \{v_{i-1}, v_{i+1}$\}. For $\varepsilon > 0$, we define a new matrix  $M'$ in which $M$ can be constructed from $M'$, by \[M' = [ \cdots | v_{i-1} | v_i - \varepsilon v_j | v_{i+1} | \cdots ].\]

    If $v_i \preceq v_j$, then there exists some $\varepsilon > 0$ sufficiently small such that $M'$ is nonnegative, and every $3 \times 3$ minor of $M'$ is nonnegative.

    In particular, because the map $M \mapsto M'$ is linear with respect to $\varepsilon$, there exists a maximal choice of such $\varepsilon$. For this maximal choice of $\varepsilon$, either some $3 \times 3$ minor of $M'$ is zero, or $v_i- \varepsilon v_j \not\preceq v_j$.
\end{lemma}

\begin{proof}
    Note that $v_i \preceq v_j$ means that for every row $k$, if $(v_j)_k \neq 0$, then $(v_i)_k \neq 0$. So for small $\varepsilon > 0$, the vector $v_i - \varepsilon v_j$ is nonnegative. Similarly, every $3 \times 3$ minor of $M$ is strictly positive. Therefore for some small $\varepsilon > 0$, every $3 \times 3$ minor of $M'$ is nonnegative.
\end{proof}

\begin{lemma}\label{lem:zero_order}
    Let $M = [v_1 | \cdots | v_n]$ be a nonnegative $m \times n$ real-valued matrix. Assume $m \geq 3$ and every three distinct columns of $M$ are linearly independent. Then the zeros in each row of $M$ are consecutive (in cyclic order). Equivalently, for any two distinct columns $v_i$,$v_j$ and row $k$, if $(v_i)_k = 0$ and $(v_j)_k = 0$, then either $(v_{i-1})_k = 0$ or $(v_{i+1})_k = 0$.
\end{lemma}

\begin{proof}
    Suppose for a contradiction that there exist four columns $v_i, v_s, v_j, v_t$ in cyclic order $i < s < j < t$, such that for some row $k$, we have $(v_i)_k = 0$, $(v_j)_k = 0$, $(v_s)_k \neq 0$, $(v_t)_k \neq 0$. We may re-scale $v_s$ and $v_t$, and assume $(v_s)_k = 1$, $(v_t)_k = 1$. 
    
    Next, since $v_i$ and $v_j$ are linearly independent, we can find two different rows $k'$, $k''$ so that the corresponding $2 \times 2$ minor of $[v_i | v_j]$ is non-zero. Using the rows $k, k', k''$, we define a $3 \times 4$ sub-matrix $[\Tilde{v}_i | \Tilde{v}_s | \Tilde{v}_j | \Tilde{v}_t]$ of $M$. If we compute determinants along row $k$, we see that \[\det([\Tilde{v}_i | \Tilde{v}_s | \Tilde{v}_j]) = - \det([\Tilde{v}_i | \Tilde{v}_j | \Tilde{v}_t]).\] Both of these determinants should be strictly positive, a contradiction.
\end{proof}

\begin{lemma}\label{lem:three_cols}
    Let $M = [v_1 | \cdots | v_n]$ be a nonnegative $3 \times n$ real-valued matrix. Assume every $3 \times 3$ maximal minor of $M$ is strictly positive and $\text{rank}(M) = 3$. If $v_i \not\preceq v_{i+1}$ and $v_i \not\preceq v_{i-1}$ for every column $v_i$ of $M$, then $M$ has three columns, i.e., $n=3$.

    In particular, either every row and column of $M$ contains exactly one zero entry, or every row and column of $M$ contains exactly one non-zero entry.
\end{lemma}

\begin{proof}
    Note that if a column $v_i$ does not contain any zero entry, then $v_i \preceq v_{j}$ for every column $v_j$ of $M$. Hence every column of $M$ must have at least one zero.
    
    If every row of $M$ has at most one zero, then $M$ can only have at most three columns, so there is nothing to show. Each column will only have one zero, because every column has at least one.
    
    If some row of $M$ has three zeros, then some $3 \times 3$ minor of $M$ would be zero, a contradiction. So it suffices to prove the lemma when some row of $M$ has exactly two zeros.
 
    Assume there is a row $k$ of $M$ such that we have two zeros in the first two columns $v_1$ and $v_2$ of $M$; that is to say $(v_1)_{k} = 0$ and $(v_2)_{k} = 0$.
    
    Since $v_2 \not\preceq v_1$, there exists some row $k'$ such that $(v_2)_{k'} = 0$ but $(v_1)_{k'} \neq 0$. Since $v_1 \not\preceq v_2$, there exists some row $k''$ such that $(v_1)_{k''} = 0$ but $(v_2)_{k''} \neq 0$. Then $k, k', k''$ are all of the rows of $M$; we have shown that both $v_1$ and $v_2$ have only one non-zero entry.

    Next, notice that $v_3 \not\preceq v_2$, so in particular, since $(v_2)_{k''} \neq 0$ is the only non-zero entry of $v_2$, we must have $(v_3)_{k''} = 0$. Since we also have $(v_1)_{k''} = 0$, it follows from lemma $\ref{lem:zero_order}$ that $v_1$ and $v_3$ must be adjacent columns in our cyclic ordering. This can only happen if $n = 3$, i.e., $M$ has three columns. Furthermore, the previous argument that shows $v_1$ and $v_2$ has only one non-zero entry can be applied here to show that $v_1$ and $v_3$ have only one non-zero entry, completing the proof.
\end{proof}

\begin{theorem}
    Let $M = [v_1 | \cdots | v_n]$ be a nonnegative $3 \times n$ real-valued matrix. Suppose $\text{rank}(M) = 3$. Then every $3 \times 3$ maximal minor of $M$ is nonnegative if and only if $M$ can be constructed from the $3 \times 3$ identity matrix $I_3$.
\end{theorem}

\begin{proof}
    First, by corollary $\ref{cor:3x3_pos}$, we can assume, without loss of generality, that every $3 \times 3$ minor of $M$ is strictly positive. If any two adjacent columns of $M$ are comparable with respect to $\preceq$, then we can repeatedly apply lemma $\ref{lem:col_subtract}$; we can construct $M'$ from some nonnegative matrix $M''$ such that every $3 \times 3$ minor of $M''$ is strictly positive and no two consecutive columns of $M$ are comparable. It then follows from lemma $\ref{lem:three_cols}$, that $M''$ is a $3 \times 3$ square matrix. Hence we have reduced the problem to the case when $M$ is a square matrix.

    So we will proceed by assuming $M$ is a nonnegative $3 \times 3$ matrix with positive determinant. In fact lemma $\ref{lem:three_cols}$ tells us that if $M$ is not already obtained from $I_3$ by re-scaling its columns, then $M$ has exactly one zero in each row and in each column. 
    
    In this case, we assume the first entry in the first column of $M$ is zero. So $(v_1)_1 = 0$, but $(v_2)_1 \neq 0$ and $(v_3)_1 \neq 0$. Since each row and column has exactly one zero, we have either $(v_2)_2 = 0$ and $(v_3)_3 = 0$, or we have $(v_2)_3 = 0$ and $(v_3)_2 = 0$. One can easily verify that in the latter case, $\det(M) < 0$, so the first case must be true. Then we can construct $M$ via
    \begin{align*}
        \begin{bmatrix}1 & 0 & 0 \\ 0 & 1 & 0 \\ 0 & 0 & 1\end{bmatrix} \mapsto \begin{bmatrix}0 & 0 & 0 & 1 & 1 & 0 \\ 1 & 0 & 0 & 0 & 0 & 1 \\ 0 & 1 & 1 & 0 & 0 & 0\end{bmatrix} \mapsto \begin{bmatrix}0 & 0 & 0 & (v_2)_1 & (v_3)_1 & 0 \\ (v_1)_2 & 0 & 0 & 0 & 0 & (v_3)_2 \\ 0 & (v_1)_3 & (v_2)_3 & 0 & 0 & 0\end{bmatrix} \mapsto M.
    \end{align*}
\end{proof}

\end{document}